\documentclass[12pt,leqno]{amsart}
\usepackage{latexsym,amssymb,amsmath,multicol, rotating,lscape}
\usepackage{letltxmacro}

\usepackage{tikz}
\usetikzlibrary{calc,decorations.markings}

 2

\newtheorem{thm}{Theorem}[section]
\newtheorem{lem}[thm]{Lemma}

\newcommand{\thmref}[1]{Theorem~\ref{#1}}
\newcommand{\lemref}[1]{Lemma~\ref{#1}}

\theoremstyle{remark}
\newtheorem{rmk}{Remark}[section]

\begin{document}

\title[Asymptotic for the power moment ]
{Average behaviour  of  Hecke eigenvalues over certain polynomial}

\author{Lalit Vaishya}
\address{Current Address: The Institute of Mathematical Sciences (A CI of Homi Bhabha
National Institute)
               CIT Campus, Taramani,
               Chennai - 600 113
               Tamilnadu 
               India.}
               \email{lalitvaishya@gmail.com, lalitv@imsc.res.in}


 
\maketitle

\begin{abstract}

In the article, we investigate the average behaviour of normalised Hecke eigenvalues over certain polynomials and establish an estimate for the power moments of the normalised Hecke eigenvalues of a normalised Hecke eigenform of weight $k \ge 2$ for the full modular group $SL_2(\mathbb{Z})$ over certain polynomial, given by a sum of triangular numbers with certain positive coefficients. More precisely, for each $r \in \mathbb{N}$, we obtain an asymptotic for the following sum
\begin{equation*}
\begin{split}
\displaystyle{\sideset{}{^{\flat }}\sum_{ \alpha(\underline{x}))+1 \le X \atop \underline{x} \in {\mathbb Z}^{4}} } \lambda_{f}^{r}(\alpha(\underline{x})+1) , \\
\end{split}
\end{equation*}
where $\displaystyle{\sideset{}{^{\flat }}\sum}$ means that the sum runs over the square-free positive  integers, and $\lambda_{f} (n)$ is the normalised $n^{\rm th}$-Hecke eigenvalue of a normalised Hecke eigenform  $f \in S_{k}(SL_2(\mathbb{Z}))$,  and   $\alpha(\underline{x}) = \frac{1}{2} \left( x_{1}^{2}+ x_{1} +  x_{2}^{2} + x_{2} + 2 ( x_{3}^{2} + x_{3}) + 4 (x_{4}^{2} + x_{4}) \right) \in {\mathbb Q}[x_{1},x_{2},x_{3},x_{4}] $ is a polynomial,  and $\underline{x} = (x_{1},x_{2},x_{3},x_{4})  \in {\mathbb Z}^{4}$. 
\end{abstract}

\section{Introduction}

Let $S_{k}(SL_2(\mathbb{Z}))$ denote the ${\mathbb C}$-vector space of cusp forms of weight $k$ for the full modular group $SL_2(\mathbb{Z})$. A cusp form $f \in S_{k}(SL_2(\mathbb{Z}))$ is said to be a Hecke eigenform if $f$ is a simultaneous eigenfunction for all the Hecke operators. Let  $a_{f}(n)$ denote the  $n^{\rm th}$ Fourier coefficient of a cusp form $f \in S_{k}(SL_2(\mathbb{Z}))$. A cusp form $f$ is said to be normalised if $a_f(1)=1$. We define the normalised $n^{\rm th}$ Fourier coefficients $ \lambda_f(n)$ given by; $ \lambda_f(n) := {a_{f}(n)}/{n^{\frac{k-1}{2}}}$.  The normalised Fourier coefficient $\lambda_{f}(n)$ is a multiplicative function and satisfies the following Hecke relation \cite[Eq. (6.83)]{Iwaniec}:
\begin{equation}
\lambda_{f}(m)\lambda_{f}(n) = \sum_{d \vert m,n }  \lambda_{f}\left(\frac{mn}{d^2}\right),  
\end{equation}
 for all positive integers $m$ and $n.$ Ramanujan conjecture predicts that $|\lambda_{f}(p)| \le 2$.  It has been established in a pioneer work of  Deligne. More precisely, we have  
\begin{equation} \label{lambda-coefficient-bound}
|\lambda_{f}(n)| \le d(n) \ll_{\epsilon} n^{\epsilon}, 
\end{equation}
for  any arbitrary small  $\epsilon > 0,$ where $d(n)$ denotes the number of positive divisors of $n$. 

The study of the average behavior of arithmetical functions attracts many researchers.   It is a well-known approach in analytic number theory to examine the moments of arithmetical functions to understand the behavior of arithmetical functions.  In this regard, Fourier coefficients of cuspidal automorphic forms (more precisely, Fourier coefficients of classical cusp forms) are one of the interesting hosts. Moreover, the average behaviour of arithmetical functions over certain sequences is riveting but quite mysterious. The randomness in the behavior of Fourier coefficients of classical cusp forms leads to many equidistribution results. There are very interesting results on the distribution of Fourier coefficients over certain sparse set of natural numbers.  For example,  Iwaniec and Kowalski \cite{HIwaniec} studied the distribution of $\{\lambda_{f}(p): p- {\rm prime}\}$ and established an analogue of prime number theorem for classical Hecke eigenforms, and  Blomer \cite{blomer} considered the distribution of the sequence $\{ \lambda_{f}(q(n)): n \in \mathbb{N}\},$ where $q(x)$ is a monic quadratic polynomial. More precisely, he proves that 
$$
\displaystyle{\sum_{n\le X}} \lambda_{f}(q(n)) \ll X^{\frac{6}{7}+\epsilon},
$$
for any $\epsilon>0.$ For a polynomial with more than one variable, the problem has been studied broadly for many arithmetic functions. A two-variables analogue of the sum studied in the work of Blomer \cite{blomer}, has been studied by  Banarjee-Pandey \cite{manish-soumya} and Acharya \cite{Ratna}.  More precisely,  they studied  the distribution of $\{\lambda_{f}(q(a,b))\}$ where $q(x,y) = x^2+y^2$, and obtained an estimate for the summatory function  $\displaystyle{\sum_{k,l \in \mathbb{Z} \atop k^2+l^2 \le X}} \lambda_{f}(q(k,l)).$   In our previous work (See \cite{ {Lalit}, {LalitV}}), we study the average behaviour of Hecke eigenvalue $\lambda_{f}(n)$, supported at the integers represented by  primitive integral positive definite binary quadratic forms of fixed negative discriminant $D$., In a joint work with M.K. Pandey \cite{Manish-Lalit}, we study the higher power moments of $\lambda_{f}(n)$, over the set of integers supported at the integers represented by primitive integral positive definite binary quadratic forms of fixed negative discriminant $D$. More  precisely, we obtain  an  estimate for the sum (for each fixed $r \in \mathbb{N}$ and sufficiently large  $X \ge 1$) 
\begin{equation*}
\begin{split}
\displaystyle{\sum_{\underline{x} \in {\mathbb Z}^{2} \atop Q(\underline{x}) \le X} (\lambda_{f}(Q(\underline{x})))^{r}}, \\
\end{split}
\end{equation*}
 where $\lambda_{f}(n)$ is the $n^{th}$  normalised Fourier coefficient of a Hecke eigenform $f$ and  $Q(\underline{x})$ is a primitive integral positive definite binary quadratic form (reduced form) of fixed negative discriminant $D$ with the class number $h(D) =1.$ This was a generalisation of the previous  result proved in   \cite{{Lalit}, {LalitV}}.  As a consequence, we established and improved previous results on the behaviour of sign change of $\lambda_{f}(n)$ in a short interval. 
 
 In this article, we consider the polynomial  
\begin{equation}\label{Poly}
\begin{split} 
 \alpha(\underline{x}) = \frac{1}{2} \left( x_{1}^{2}+ x_{1} +  x_{2}^{2} + x_{2} + 2 ( x_{3}^{2} + x_{3}) + 4 (x_{4}^{2} + x_{4}) \right) \in {\mathbb Q}[x_{1},x_{2},x_{3},x_{4}],
 \end{split}
\end{equation}
   and study the power moment of Hecke eigenvalues $\lambda_{f}(n)$'s over the integers which are represented by $\alpha(\underline{x})$. 

\smallskip
 We fix a few more notations and state our results. 

Let $M_{k}(\Gamma_{0}(N), \chi)$ denote the space of modular forms of weight $k$ for the congruence subgroup $\Gamma_{0}(N)$ with the nebentypus $\chi$ (a character defined on the congruence subgroup $\Gamma_{0}(N)$). For the characters $\chi$ and  $\psi$ of modulus $N$, we define the generalised divisor function given by
 \begin{equation*}
\begin{split}
\sigma_{k; \chi, \psi}( n)  = \displaystyle{\sum_{ d \mid n}} \psi(d) \chi(n/d) d^{k}. 
\end{split}
\end{equation*}
Let  ${E}_{k; \chi, \psi}$ denote an Eisenstein series in $M_{k}(\Gamma_{0}(N), \chi \psi)$ given by
 \begin{equation*}
\begin{split}
{E}_{k; \chi, \psi} (\tau) = \displaystyle{\sum_{n = 0}^{\infty}} \sigma_{k-1; \chi, \psi} ( n) q^{n}. 
\end{split}
\end{equation*}

\smallskip
Let  $\alpha(\underline{x})  \in {\mathbb Q}[x_{1},x_{2},x_{3},x_{4}] $ be  a polynomial given in \eqref{Poly}. The polynomial $\alpha(\underline{x})$ takes positive integer values at the integral point of $\mathbb{Q}^{4}$. Let $\delta_{4}(\alpha; n)$ denote the number of integral representations of a positive integer $n$ represented by a polynomial $\alpha$, i.e., 
$$
\delta_{4}(\alpha; n) =  \# \{\underline{x} \in  {\mathbb Z}^{4} : n= \alpha(\underline{x})  \}.
$$
Associated to a polynomial  $\alpha(\underline{x})$, we define the generating function $T(\tau)$ of  $\delta_{4}(\alpha; n)$  given by
\begin{equation*}
\begin{split}
T(\tau) & := \displaystyle{\sum_{\underline{x} \in {\mathbb Z}^{4} } q^{(\alpha(\underline{x}))}}  \quad = \quad  \displaystyle{\sum_{n = 0}^{\infty} \delta_{4}(\alpha; n) q^{n}}, \qquad   \quad ~~~ q = e^{2 \pi i \tau}, \tau \in \mathbb{H} \\
\end{split}
\end{equation*} 
where $\mathbb{H}$ denotes the complex upper plane. The function $q T(\tau)$ is a modular form of weight $2$ and  level $8$ with nebentypus $\chi_{8}$ \cite[Remark 1.2]{Ramki-Lalit}, where $\chi_{8}$ is a character given by Jacobi symbol, $\chi_{8} (n) := \left( \frac{8}{n} \right)$. It is well-known that  the modular space $M_{2}(\Gamma_{0}(8), \chi_{8})$ is generated by the generalised Eisenstein series ${E}_{2; \chi_{8}, \textbf{1}}$ and ${E}_{2; \textbf{1},\chi_{8}}$ where $\textbf{1}$ denotes the trivial character of modulus $8$. Then,  it is easy to see that  
$ q T(\tau) =  {E}_{2; \chi_{8}, \textbf{1}} (\tau).$ By comparing the $n^{\rm th}$ Fourier coefficients,
we have 
 \begin{equation}\label{RepF}
\begin{split}
\delta_{4}(\alpha; n-1) = \sigma_{1; \chi_{8}, \textbf{1}}(n) = \displaystyle{\sum_{ d \mid n}} ~ \chi_{8}(d) ~\frac{n}{d}.
\end{split}
\end{equation}
\begin{rmk}
Let 
$\alpha_{1}(\underline{x}) = x_{1}^{2} +2x_{2}^{2}  +  2 ( x_{3}^{2} + x_{3}) + 2 (x_{4}^{2} + x_{4}) \in {\mathbb Z}[x_{1},x_{2},x_{3},x_{4}] $ and \linebreak  $\alpha_{2}(\underline{x}) = x_{1}^{2} +  ( x_{2}^{2} + x_{2}) +  ( x_{3}^{2} + x_{3}) + 2 (x_{4}^{2} + x_{4}) \in {\mathbb Z}[x_{1},x_{2},x_{3},x_{4}] $ be the polynomials.
From the theory of modular forms, it is easy to see that the number of integral representations $R_{4}(\alpha_{1}; n-1)$ (resp. $R_{4}(\alpha_{2}; n-1)$) of a positive integer $n$ represented by the  polynomial $\alpha_{1}$ (resp. $\alpha_{2}$) is given by
 \begin{equation*}
\begin{split}
R_{4}(\alpha_{1}; n-1) & = \displaystyle{\sum_{ d \mid n}} ~ \chi_{8}(d) ~\frac{n}{d} \qquad \text{and} \qquad
R_{4}(\alpha_{2}; n-1)  = \displaystyle{\sum_{ d \mid n}} ~ \chi_{8}(d) ~\frac{n}{d}.
\end{split}
\end{equation*}
\end{rmk}
\noindent
 Let $\lambda_{f} (n)$ denote  the normalised $n^{\rm th}$-Hecke eigenvalue of a normalised Hecke eigenform \linebreak   $f \in S_{k}(SL_2(\mathbb{Z}))$.   For  each fixed $  r \in \mathbb{N}$ and $X \ge 1$, we define the following power sum:
\begin{equation}\label{Sum-PIBQF}
\begin{split}
S_{r}(X) := \displaystyle{\sideset{}{^{\flat }}\sum_{ \alpha(\underline{x})) +1 \le X \atop \underline{x} \in {\mathbb Z}^{4}} } \lambda_{f}^{r}(\alpha(\underline{x})+1) \\
\end{split}
\end{equation}
where  $\alpha(\underline{x}) \in {\mathbb Q}[x_{1},x_{2},x_{3},x_{4}] $ is a polynomial defined in \eqref{Poly}. 

\bigskip
With these notations, we state our results.
\begin{thm}\label{PolyEst}
Let  $\epsilon >0$ be an arbitrarily small. For sufficiently large $X$,  we have
\begin{equation} \label{SR-Est1}
\begin{split}
S_{1}(X) &=  O_{f,D,\epsilon}(X^{\frac{3}{2}+\epsilon})
\end{split}
\end{equation}
and
\begin{equation} \label{SR-Est2}
\begin{split}
S_{2}(X) &= C X^{2}  + O(X^{\frac{8}{5}+\epsilon})\\
\end{split}
\end{equation}
where $C$ is a positive absolute constant.
\end{thm}

\begin{thm}\label{Estimate-TR}
Let $\epsilon >0$ be an arbitrarily small.  For each $ r \ge 3$  and sufficiently large $X$,  we have the following estimates for $S_{r}(X)$.
\begin{equation} \label{SR-Est}
\begin{split}
S_{r}(X) = X^{2} P_{r}(\log X) + O(X^{2- \frac{1}{2(1+ \gamma_{r})}+\epsilon}),\\
\end{split}
\end{equation}
where for each $r = 2m$ $(m \ge 2),$  $P_{r}(t)$ is polynomial of degree $ d_{r} = \frac{1}{m} {r \choose m} -1 $ and
$$\gamma_{r} = \frac{13}{82m} {2m \choose m-1 } + \frac{15}{8(m-1)} {2m \choose m-2 } + \frac{1}{4} \left[ \displaystyle{\sum_{n=0}^{m-2}\frac{(2m-2n+1)^{2}}{n} {2m \choose n-1}} \right],$$
and for each $r = 2m+1$ $(m \ge 1)$,  $P_{r}(t) \equiv 0$ and with 
$$\gamma_{r} =  \frac{2}{3m} {2m+1 \choose m-1 }  + \frac{1}{4} \left[ \displaystyle{\sum_{n=0}^{m-1}\frac{(2m+1-2n+1)^{2}}{n} {2m+1 \choose n-1}} \right]-\frac{5}{6}.$$

 \end{thm}
 
 \begin{rmk}
 One can obtain exactly same result as in \thmref{PolyEst} and \thmref{Estimate-TR} for the polynomial $\alpha_{1}$ and $\alpha_{2}$ in place of $\alpha$.
\end{rmk}
Throughout the paper, $\epsilon$ denotes an arbitrarily small positive constant but not necessarily the same one at each place of occurrence.

\section{Key ingredients}

The sums defined in \eqref{Sum-PIBQF} can be expressed in terms of known arithmetical functions using \eqref{RepF}, i.e.,
\begin{equation}\label{Sum-T} 
\begin{split}
S_{r}(X) &= \displaystyle{\sideset{}{^{\flat }}\sum_{ \alpha(\underline{x}))+1 \le X \atop \underline{x} \in {\mathbb Z}^{4}} } \lambda_{f}^{r}(\alpha(\underline{x})+1) =  \displaystyle{\sideset{}{^{\flat }} \sum_{n \le X }} \left( \lambda_{f}^{r}(n) \left(\sum_{ n= \alpha(\underline{x}) +1 } 1\right) \right) =   \displaystyle{\sideset{}{^{\flat }} \sum_{n \le X }} \lambda_{f}^{r}(n)\delta_{4}(\alpha; n-1)  \\
& =   \displaystyle{\sideset{}{^{\flat }} \sum_{n \le X }} \lambda_{f}^{r}(n)  \sigma_{1; \chi_{8}, \textbf{1}}(n) \\
\end{split}
\end{equation}
We define the following Dirichlet series associated to the sum $S_{r}(X) $ given by: 
\begin{equation}\label{R-L-function}
\begin{split}
R_{r}(s) &=  \displaystyle{\sideset{}{^{\flat }} \sum_{n \ge 1}} ~~ \frac{\lambda_{f}^{r}(n)\sigma_{1; \chi_{8}, \textbf{1}}(n)}{n^{s}}.
\end{split}
\end{equation}
The above Dirichlet series converges for $\Re(s) >2$. We obtain an estimate for $S_{r}(X)$ by using the decomposition of  $R(s)$ in terms of known $L$-functions associated to Hecke eigenform $f \in S_{k}(SL_2(\mathbb{Z}))$. Before acquiring  the decomposition of $R(s)$, we define the $L$-functions associated to a normalised Hecke eigenform $f(\tau) = \displaystyle{\sum_{n=1}^\infty \lambda_f(n)n^{\frac{k-1}{2}}q^{n}} \in S_{k}(SL_2(\mathbb{Z})).$ 
 The Hecke $L$-function  associated to $f$ is given by ($ \Re(s) >1$)
\begin{equation}
\begin{split}
L(s, f) &= \sum_{n \ge 1} \frac{\lambda_{f}(n)}{n^{s}} =  \prod_{p} \left(1-\frac{\lambda_f(p)}{p^s}-\frac{1}{p^{2s}}\right)^{-1}  = \prod_{p} \left(1-\frac{\alpha_p}{p^s}\right)^{-1}\left(1-\frac{\beta_p}{p^{s}}\right)^{-1},
\end{split}
\end{equation}
where $\alpha_p+\beta_p=\lambda_f(p)$ and $\alpha_p\beta_p=1.$  For a given Dirichlet character $\chi$ of modulus $N,$ the twisted Hecke $L$- function  is defined as follows:
\begin{equation}
\begin{split}
L(s, f \times \chi ) &= \sum_{n \ge 1} \frac{\lambda_{f}(n) \chi(n)}{n^{s}}  \qquad \Re(s) >1. 
\end{split}
\end{equation}
The twisted Hecke $L$-function $L(s, f \times \chi )$ is associated to the cusp form $f _{\chi} \in S_{k}(\Gamma_{0}(N^{2}))$ with Fourier coefficients $ \lambda_f(n) \chi(n)$.  Both the $L$-functions satisfy a nice functional equation and it has analytic continuation to whole $\mathbb{C}$-plane \cite[Section 7.2]{Iwaniec}.

\smallskip
For $m \ge 2 ,$ the $m^{th}$ symmetric power $L$-function is defined as
\begin{equation}\label{Symf}
\begin{split}
L(s,sym^{m}f)&: = \prod_{p-{\rm prime}} \prod_{j=0}^{m} \left(1-{\alpha_p}^{m-j}{\beta_p}^{j} {p^{-s}}\right)^{-1} 
                      =\sum_{n=1}^\infty \frac{\lambda_{sym^{m}f}(n)}{n^s},
\end{split}
\end{equation}
where ${\lambda_{sym^{m}f}(n)}$ is multiplicative arithmetical function. At prime values, it is given by 
\begin{equation}\label{SymLf}
\begin{split}
{\lambda_{sym^{m}f}(p)}  =  \lambda_{f}(p^{m}).  
\end{split}
\end{equation}
 From Deligne's bound, we have 
$$
|{\lambda_{sym^{m}f}(n)}| \le d_{m+1}(n) \ll_{\epsilon} n^{\epsilon}
$$
for any real number $\epsilon >0$ and $d_{m}$ denotes the $m$-fold divisor function.

\smallskip
For each $m \ge 2,$ we also define the twisted $m^{th}$ symmetric power $L$-functions given by
\begin{equation}
\begin{split}
L(s, sym^{m}f \times \chi ) &:= \sum_{n \ge 1} \frac{\lambda_{sym^{m}f}(n) \chi(n)}{n^{s}}, 
\end{split}
\end{equation}
similar to the twisted Hecke $L$-function. these $L$-functions are automorphic (for details, see \cite{{Newton}, {NewThorne}}) and  inherit the property similar to the Hecke $L$-function.  For a  holomorphic Hecke eigenform $f$, J. Cogdell and P. Michel \cite{Cog-Mic}  have given the explicit description of analytic continuation and functional equation for the function $L(s,sym^{m}f)$, $ m \ge 3$.  

\smallskip
Let  $\zeta(s)$ and $L(s, \chi) $ (for a Dirichlet character $\chi$ of modulus $N$) denote the Riemann zeta function and Dirichlet $L$-function, respectively defined by 
\begin{equation}\label{RDFun}
\begin{split} 
 \zeta(s) = \displaystyle{\sum_{n \ge 1} n^{-s}} \quad {\rm and} \quad  L(s, \chi) = \displaystyle{\sum_{n \ge 1}\chi(n) n^{-s}}.
 \end{split} 
\end{equation}
We assume the following conventions:
\begin{equation*}
\begin{split}
\begin{cases}
L(s, sym^{0}f) &= \zeta(s), \qquad \qquad L(s, sym^{0}f \times \chi ) = L(s, \chi), \\
L(s, sym^{1}f  ) &= L(s, f), \qquad \quad L(s, sym^{1}f \times \chi )  = L(s, f \times \chi). \\
  \end{cases}
\end{split} 
\end{equation*}

 With these definitions,  we state the decomposition of $R_{r}(s )$, $ r \in \mathbb{N}  $ into well-known $L$-functions.
 \begin{lem}\label{LDecomposition}
Let  $ r \in \mathbb{N}.$ We have the following decomposition for $R_{r}(s)$.
 \begin{equation} \label{DecompositionL}
\begin{split}
 R_{r}(s) & = L_{r}(s) \times U_{r}(s),  \quad \quad {\rm where } \\
 \end{split}
\end{equation}
\begin{equation*}
\begin{split}
 L_{1}(s)  &=  L(s-1, f) L(s, f\times \chi_{8}) \\
 L_{2}(s)  &=   \zeta(s-1) L(s, \chi_{8}) L(s-1, sym^{2}f)  L(s, sym^{2}f \times \chi_{8}).
\end{split}
\end{equation*}
and for each $r \ge 3$,
 \begin{equation*}
\begin{split}
L_{r}(s) & =   \prod_{n=0}^{[r/2]} \left({L(s-1, sym^{r-2n}f)}^{\left({r \choose n}- {r \choose {n-1}}\right)}  L(s, sym^{r-2n}f \times \chi_{8})^{\left({r \choose n}- {r \choose {n-1}}\right)} \right),       \\
\end{split}
\end{equation*}
and ${r \choose n }$ is the binomial coefficient with the convention ${r \choose n } =0$ if $n<0,$ and $\chi_{8}$ is the Dirichlet character modulo $8$ and $U_{r}(s)$ is a Dirichlet series given by 
$$
  U_{r}(s) =  \displaystyle{\prod_{p} \left(1  + \frac{{(  A(p^{2}) - \lambda_{f}(p)}^{2r} \sigma_{1; \chi_{8}, \textbf{1}}^{2}(p))}{p^{2s}} + \cdots \right)}.\\
$$
It converges absolutely and uniformly for $\Re(s)>\frac{3}{2}$ and $U_{r}(s) \neq 0$ for $\Re(s)=2.$
\end{lem}

Before proving \lemref{LDecomposition}, we state the following result which explicitly governs the proof of  \lemref{LDecomposition}.
\begin{lem}\cite[Lemma 2.2]{{Manish-Lalit}}\label{ChebpolyLem}
Let $\ell \in \mathbb{N}.$ For each $j$ with $0 \le j \le \ell$ and  $j \equiv \ell \pmod 2$, let $A_{\ell,j} := {\ell \choose {\frac{\ell-j}{2}}} - {\ell \choose {\frac{\ell-j}{2}-1}}$ and $0$ otherwise and  $T_{m}(2x):= U_{m}(x)$ where $U_{m}(x)$ is the $m^{\rm th}$ Chebyshev polynomial of second kind.
Then 
\begin{equation*}
\begin{split}
x^{\ell} &= \sum_{j=0}^{\ell} A_{\ell,j} T_{\ell-j}(x). \\
\end{split}
\end{equation*}
\end{lem}

 \subsection{Proof of \lemref{LDecomposition}}

From Deligne's estimate, we know that  $\lambda_{f}(p) = 2\cos \theta$, and  $\lambda_{f}(p^{r}) = T_{r} (2 \cos \theta) =U_{r} ( \cos \theta) $. From \lemref{ChebpolyLem}, we get an  expression for ${\lambda_{f}(p) }^{r}$  in terms of the $j^{th}$- symmetric power Fourier coefficient $\lambda_{sym^{j}f}(p)$, i.e.,
\begin{equation}\label{FCRel}
\begin{split}
{\lambda_{f}(p)}^{r}    = \left(\sum_{n=0}^{r/2} \left( {r \choose n} - {r \choose {n-1}} \right) \lambda_{sym_{f}^{r-2n}}(p) \right).\\ 
\end{split}
\end{equation}
We know that  $\lambda_{f}(n)$ and $ \sigma_{1; \chi_{8}, \textbf{1}}(n)$  are multiplicative functions. So,  $ R_{r}(s)$ is given in terms of  an Euler product, i.e.,
 \begin{equation*}
\begin{split}
R_{r}(s) = \displaystyle{\sideset{}{^{\flat }} \sum_{n \ge 1}} ~~ \frac{(\lambda_{f}(n))^{r} \sigma_{1; \chi_{8}, \textbf{1}}(n)}{n^{s}}
 = \displaystyle{\prod_{p} \left(\!1+ \frac{(\lambda_{f}(p))^{r}  \sigma_{1; \chi_{8}, \textbf{1}}(p)}{p^s} \right)} 
 \end{split}
\end{equation*}
for $\Re(s)>2$. From \eqref{FCRel}, we have
\begin{equation*}
\begin{split}
&\lambda_{f}^{r} (p) \sigma_{1; \chi_{8}, \textbf{1}}(p) =  {\lambda_{f}^{r}(p)} (p + \chi_{8}(p))   = \left(\sum_{n=0}^{r/2} \left( {r \choose n} - {r \choose {n-1}} \right) \lambda_{sym_{f}^{r-2n}}(p) \right)(p + \chi_{8}(p)) \\ 
&\qquad \quad = \sum_{n=0}^{r/2} \left( \!{r \choose n} - {r \choose {n-1}} \! \right) p \lambda_{sym_{f}^{r-2n}}(p)  +  \sum_{n=0}^{r/2} \left( \! {r \choose n} - {r \choose {n-1}} \! \right) \lambda_{sym_{f}^{r-2n}}(p)\chi_{8}(p). 
\end{split}
\end{equation*}
For $\Re(s) >2$, we express the function
 $$L_{r}(s) =  \prod_{n=0}^{[r/2]} \left({L(s-1, sym^{r-2n}f)}^{\left({r \choose n}- {r \choose {n-1}}\right)}  L(s, sym^{r-2n}f \times \chi_{8})^{\left({r \choose n}- {r \choose {n-1}}\right)} \right)$$
 as an Euler product of the form  
 \begin{equation*}
\begin{split}
 \displaystyle{\prod_{p} \left(1+ \frac{A(p)}{p^s} + \frac{A(p^{2})}{p^{2s}} + \cdots \right) }, \quad {\rm where } \quad A(p) = -{\lambda_{f}(p)}^{r} \sigma_{1; \chi_{8}, \textbf{1}}(p). 
 \end{split}
\end{equation*}
Moreover, for each prime $p$,  we define the sequence $B(p) = 0$, for each $r \ge 2$, \linebreak  $B(p^{r}) = A(p^{r}) + A(p^{r-1}) \lambda_{f}^{r} (p) \sigma_{1; \chi_{8}, \textbf{1}}(p).$ It is easy to see that $B(n) \ll n^{1+\epsilon}$ for any $\epsilon$. Associated to this sequence, We define the Euler product for $U_{r}(s) $  given by
 \begin{equation*}
\begin{split}
U_{r}(s) =  \displaystyle{\prod_{p} \left(1+ \frac{B(p)}{p^s} + \frac{B(p^{2})}{p^{2s}} + \cdots \right) }.
 \end{split}
\end{equation*}
Then, it is easy to see that  
 \begin{equation*}
\begin{split}
R_{r}(s) =  L_{r}(s)  U_{r}(s).
\end{split}
\end{equation*}
This completes the proof.

\subsection{Convexity bound and integral  moment  of $L$-functions}
\begin{lem}\label{Riemann Zeta }
Let $\zeta(s)$ be the Riemann zeta function. Then for any $\epsilon >0$, we have 
\begin{equation}
\begin{split}
\zeta(\sigma+it) &\ll_{\epsilon} (1+|t|)^{{{\rm max} \left\{ \frac{13}{42}(1-\sigma), 0 \right\} }+\epsilon}
\end{split}
\end{equation}
uniformly for $\frac{1}{2} \le \sigma \le 1$ and $|t| \ge 1$ and 
\begin{equation}
\begin{split}
\int_{0}^{T} \left|\zeta\left(\frac{1}{2}+it \right)\right|^{2} dt &\ll_{\epsilon} T^{1+\epsilon} \\
\end{split}
\end{equation} 
uniformly for $T \ge 1.$ For sub-convexity bound and integral estimate  of $\zeta(s)$, we refer to \cite[Theorem  5]{Bourgain} and \cite[Theorem 8.4]{Ivic} respectively.
\end{lem}
 
 \begin{lem} \cite[eq. (1.1)]{HBrown}\label{Dirichlet L function }
Let $L(s, \chi)$  be the Dirichlet $L$-function for a  Dirichlet character $\chi$ modulo N	. Then for any $\epsilon >0$, we have 
\begin{equation}
\begin{split}
L(\sigma+it, \chi) &\ll_{\epsilon, N} (1+|t|)^{\frac{1}{3}(1-\sigma)+\epsilon}
\end{split}
\end{equation}
uniformly for $\frac{1}{2} \le \sigma \le 1$ and $|t| \ge 1$.
\end{lem}

\begin{lem}\label{Modular L function }
For any $\epsilon >0$, the sub-convexity bound  of $L(s, f)$ is given by
\begin{equation}
\begin{split}
L(\sigma+it, f) &\ll_{f,\epsilon} (1+|t|)^{{{\rm max} \left\{\frac{2}{3}(1-\sigma), 0\right\} }+\epsilon}
\end{split}
\end{equation}
uniformly for $\frac{1}{2} \le \sigma \le 1$ and $|t| \ge 1,$ and the second  integral moment of $L(s, f)$ is given by
\begin{equation}\label{Secondmoment}
\begin{split}
\int_{0}^{T} \left| L\left(\frac{1}{2} + \epsilon+it , f \right)\right|^{2} dt &\ll_{f,\epsilon} T^{1+\epsilon} \\
\end{split}
\end{equation} 
uniformly for $T \ge 1.$ The results also hold for $f\ \times \chi$ in place of $f$ with a different  the absolute constant depends on $f$ and $\epsilon$.
\end{lem}
\begin{proof}
The sub-convexity bound of Hecke $L$-function $L(s, f)$ follows from the standard argument of Phragmen - Lindel\"{o}f convexity principle and a result of A. Good  \cite[Corollary]{AGood}. For the integral estimate, we refer to \cite[Theorem 2]{AIvic}.  
\end{proof}

 \begin{lem}\cite[Corollary 2.1]{Nunes}
 For any arbitrarily small $\epsilon >0$, we have 
\begin{equation}
\begin{split}
L(\sigma+it, sym^{2}f)  &\ll_{f, \epsilon} (1+|t|)^{{{\rm max} \left\{\frac{5}{4}(1-\sigma), 0\right\} }+\epsilon}
\end{split}
\end{equation}
uniformly for $\frac{1}{2} \le \sigma \le 1$ and $|t| \ge 1$.  A similar result also holds for  $sym^{2}f \otimes \chi$.
\end{lem}

 \begin{lem}\label{General L fun-Con } \cite[pp. 100]{HIwaniec}
Let $L(s,F)$ be an $L$-function of degree $m \ge 2,$ i.e.,
\begin{equation}
\begin{split}
L(s, F) & = \sum_{n \ge 1} \frac{\lambda_{F}(n)}{n^{s}} = \prod_{p-{\rm prime}} \prod_{j= 1}^{m} \left(1-\frac{\alpha_{p, f, j}}{p^s}\right)^{-1},
\end{split}
\end{equation}
where $\alpha_{p, f, j}$, $1 \le j \le m$; are the local parameters of $L(s, F)$ at prime $p$ and $\lambda_{F}(n) = O(n^{\epsilon})$ for any $\epsilon>0.$ We assume that the series and Euler product converge absolutely for $\Re(s) > 1$ and $L(s, F)$ is an entire function except possibly for a pole at $s = 1$ of order $r$ and satisfies a nice functional equation $(s \rightarrow 1-s)$. Then for any $\epsilon >0$ and $s =\sigma+it$, we have 
\begin{equation}
\begin{split}
\left( \frac{s-1}{s+1}\right)^{r}L(\sigma+it, F) &\ll_{F,\epsilon} (1+|t|)^{\frac{m}{2}(1-\sigma)+\epsilon}
\end{split}
\end{equation}
uniformly for $0 \le \sigma \le 1$.  and $|t| \ge 1$. 
\end{lem}

\begin{lem}\label{Mean-value} \cite[Lemma 2.6]{G. Lu1}
Let $L(s, F)$ be an $L$-function of degree $m \ge 2$. Then for any $\epsilon >0$ and  $T\ge 1$, We have 
\begin{equation}
\begin{split}
\int_{T}^{2T} \left|L \left(\frac{1}{2} + \epsilon+it, F\right)\right|^{2} dt &\ll_{F,\epsilon} T^{\frac{m}{2}+\epsilon}.
\end{split}
\end{equation}
\end{lem}

\section{Proof of results}
\subsection{\textbf{General philosophy:}}
Let $1 \le Y < \frac{X}{2}$. In order to obtain an upper bound for the sum $S_{r}(X )$ given in \eqref{Sum-PIBQF}, we introduce a smooth compactly supported function $w(x)$ satisfying; $w(x) =1$ for $x \in [2Y, X],$ $w(x) = 0$ for $x<Y$ and $x> X+Y,$ and $w^{(r)}(x) \ll_{r} Y^{-r}$ for all $r\ge 0.$ In general, for any arithmetical function $f(n), $ we have 
\begin{equation}\label{UpperEstimate}
\sum_{n \le X} f(n) = \sum_{n = 1}^{\infty} f(n)w(n) +  O\left(\sum_{n< 2Y} |f(n)| \right) + O\left(\sum_{X< n< X+ Y} |f(n)| \right). 
\end{equation}
Moreover, by Mellin's inverse transform, we have 
\begin{equation}
 \sum_{n = 1}^{\infty} f(n)w(n) = \frac{1}{2 \pi i} \int_{(b)} \tilde w(s) \left(\sum_{n\ge 1} \frac{f(n)}{n^{s}}\right)  ds, 
 \end{equation}
where $b$ is a real number larger than the  abscissa of absolute convergence of $\displaystyle{\sum_{n\ge 1} {f(n)}{n^{-s}}}$ and the Mellin's transform $\tilde w(s)$ is given by following integral:
$
\tilde w(s) = \int_{0}^{\infty} w(x) x^{s} \frac{dx}{x}.
$

We observe that (due to integration by parts), 
\begin{equation}\label{FourierW}
\tilde w(s) =  \frac{1}{s(s+1)\cdots(s+m-1)}\int_{0}^{\infty} w^{(m)}(x) x^{s+m-1} dx \ll \frac{Y}{X^{1-\sigma}}  \left(\frac{X}{|s|Y}\right)^{m},
\end{equation}
for any $m\ge 0,$ where $\sigma = \Re(s).$ For details, we refer to \cite[Section 3]{YJGL}.

\smallskip 
From Equation \eqref{UpperEstimate} with $f(n) =  \lambda_{f}^{r}(n)  \sigma_{1; \chi_{8}, \textbf{1}}(n),$ we have 
\begin{equation*}
\begin{split}
\sideset{}{^{\flat }} \sum_{n \le X}  \lambda_{f}^{r}(n)  \sigma_{1; \chi_{8}, \textbf{1}}(n) & = \sideset{}{^{\flat }} \sum_{n \ge 1} \lambda_{f}^{r}(n)  \sigma_{1; \chi_{8}, \textbf{1}}(n) w(n)  \\
& \quad  +  O\left(\sideset{}{^{\flat }} \sum_{n < 2Y}  |\lambda_{f}^{r}(n)  \sigma_{1; \chi_{8}, \textbf{1}}(n)| \right) + O\left(\sideset{}{^{\flat }} \sum_{X< n< X+ Y} | \lambda_{f}^{r}(n)  \sigma_{1; \chi_{8}, \textbf{1}}(n)| \right). 
\end{split}
\end{equation*}
Moreover, by Mellin's inverse transform, we have 
\begin{equation}\label{Perron-Type}
\sideset{}{^{\flat }} \sum_{n \ge 1}  \lambda_{f}^{r}(n)  \sigma_{1; \chi_{8}, \textbf{1}}(n)  w(n)  = \frac{1}{2 \pi i} \int_{(b)} \tilde w(s) R_{r}(s)  ds, 
 \end{equation}
where $R_{r}(s)$ is defined in \eqref{R-L-function} and  $b= 2+ \epsilon$ for some arbitrarily small $ \epsilon>0.$ Since   $ \lambda_{f}^{r}(n)  \sigma_{1; \chi_{8}, \textbf{1}}(n)  \ll n^{1+\epsilon}$ for any $\epsilon>0$. Hence
\begin{equation}\label{Errorbound}
\begin{split}
  O\left(\sideset{}{^{\flat }} \sum_{~n < 2Y} | \lambda_{f}^{r}(n)  \sigma_{1; \chi_{8}, \textbf{1}}(n) | \right) + O\left(\sideset{}{^{\flat }} \sum_{~X< n< X+ Y  } \!\!\!\! | \lambda_{f}^{r}(n)  \sigma_{1; \chi_{8}, \textbf{1}}(n)| \right) = Y^{2+\epsilon} + X^{1+\epsilon}Y \ll X^{1+\epsilon}Y.
\end{split}
\end{equation}
We shift the line of integration from $\Re(s) = 2+\epsilon$ to $\Re(s) = \frac{3}{2}+\epsilon$ and apply Cauchy's residue theorem to get 
\begin{equation}\label{RHSInt}
\begin{split}
\sideset{}{^{\flat }} \sum_{n \ge 1}  \lambda_{f}^{r}(n)  \sigma_{1; \chi_{8}, \textbf{1}}(n)  w(n)  &= \frac{1}{2 \pi i} \int_{(2 + \epsilon)} \tilde w(s) R_{r}(s)  ds  \\
& =   \underset{s= 2}{\rm Res} \left(R_{r}(s ) \tilde w(s)  \right)  +  \frac{1}{2 \pi i} \int_{(3/2+\epsilon)} \tilde w(s) R_{r}(s)  ds + O(X^{-A})
\end{split}
 \end{equation}
 for sufficiently large positive constant $A$. The residue term exists only when $r$ is an even positive integer otherwise 0. Since $\tilde w(s)  \ll \frac{Y}{X^{1-\sigma}}  \left(\frac{X}{|s|Y}\right)^{m}$
for any $m \ge 0,$ so  the contribution for the integral over  $|s| \ge T = \frac{X^{1+\epsilon}}{Y}$ on the right hand side of \eqref{RHSInt} is negligibly small, i.e., $O(X^{-A})$ for any large $A>0$ if one chooses sufficiently large $m>0.$ Hence, we have 
\begin{equation*}
\begin{split}
  \frac{1}{2 \pi i} \int_{(3/2+\epsilon)} \tilde w(s) R_{r}(s )  ds  & =   \frac{1}{2 \pi i} \int_{3/2+\epsilon-iT}^{3/2+\epsilon+iT} \tilde w(s)R_{r}(s )ds + O(X^{-A}) \\
  & \ll \int_{-T}^{T} |\tilde w(3/2+\epsilon+it)| |R_{r}(3/2+\epsilon+it )| dt + O(X^{-A}) \\
  & \ll  \int_{0}^{T} \frac{X^{\frac{3}{2}+\epsilon}}{|\frac{3}{2}+\epsilon+it|} |R_{r}(3/2+\epsilon+it )| dt + O(X^{-A}) \\
  & \ll \left(\int_{0}^{1} + \int_{1}^{T}\right) \frac{X^{\frac{3}{2}+\epsilon}}{|\frac{3}{2}+\epsilon+it|} |R_{r}(3/2+\epsilon+it )|  dt + O(X^{-A}), \\
  \end{split}
 \end{equation*}
where the estimate in the last lines is obtained by substituting the bound for $\tilde w(s)$ (given in \eqref{FourierW}) when $m=1.$  We substitute the decomposition of  $R_{r}(s )$ ($R_{r}(s )  = L_{r}(s )  U_{r}(s ) $) from  \lemref{LDecomposition},  and  utilize the absolute convergence of $U_{r}(s )$ in the region  $\Re(s)> \frac{3}{2}$   to get 
\begin{equation}\label{IntE}
\begin{split}
  \frac{1}{2 \pi i} \int_{(3/2+\epsilon)} \tilde w(s) R_{r}(s )  ds  
  & \ll X^{\frac{3}{2}+\epsilon} + X^{\frac{3}{2}+\epsilon}  \int_{1}^{T} \frac{|L_{r}(3/2+\epsilon+it )|}{t}  dt.
  \end{split}
 \end{equation}
 Thus, combining all the estimates, we have  (for each fixed $r \in \mathbb{N}$)
 \begin{equation}\label{Gen-formula}
\begin{split}
\sideset{}{^{\flat }} \sum_{n \le X}  \lambda_{f}^{r}(n)  \sigma_{1; \chi_{8}, \textbf{1}}(n) 
& =  \underset{s= 2}{\rm Res} \left(R_{r}(s ) \tilde w(s)  \right)  + O\left(X^{\frac{3}{2}+\epsilon}  \int_{1}^{T} \frac{|L_{r}(3/2+\epsilon+it )|}{t}  dt \right)  \\
 & \quad +  O( X^{\frac{3}{2}+\epsilon})  + O( X^{1+\epsilon}Y) +  O(X^{-A})  
\end{split}
\end{equation}
where $T = \frac{X^{1+\epsilon}}{Y}$ and $Y$ be a suitable parameter, and the first term on RHS of \eqref{Gen-formula}  exists only when $r$ is even. So, it is enough to get an estimate for the integral $I_{r}$ (say) appearing in \eqref{Gen-formula}  to get the required estimate for the sum $S_{r}(X)$ defined in \eqref{Sum-PIBQF}.

\subsection{Proof of \thmref{PolyEst}:} 
Since the function $R_{1}(s) = L(s-1, f) L(s, f\times \chi_{8}) U_{1}(s)$ is holomorphic  at $s=2$. So, from \eqref{Gen-formula}, we have  
 \begin{equation}\label{First result}
\begin{split}
S_{1}(X) & =  O\left(X^{\frac{3}{2}+\epsilon}  \int_{1}^{T} \frac{ |L(1/2+\epsilon+it, f) L(3/2+\epsilon+it, f\times \chi_{8})|}{t}  dt \right) \\
& \quad  +  O( X^{\frac{3}{2}+\epsilon})   +  O( X^{1+\epsilon}Y) +  O(X^{-A}).  
\end{split}
\end{equation}
Following the argument using the dyadic division method and then the Cauchy-Schwarz inequality, we have 
\begin{equation*}
\begin{split}
& \int_{1}^{T} \frac{ |L(1/2+\epsilon+it, f) L(3/2+\epsilon+it, f\times \chi_{8})|}{t}  dt  \ll  \int_{1}^{T} \frac{ |L(1/2+\epsilon+it, f)|}{t}  dt \\
& \ll \log T \underset{1 \le T_{1} \le T} {\rm sup} \left(
\frac{1}{T_{1}}  \left( \int_{1}^{T_{1}}  {|L(1/2+\epsilon+it , f)|}^{2}  dt\right)^{\frac{1}{2}} \times  \left( \int_{1}^{T_{1}}   dt\right)^{\frac{1}{2}}\right)  \ll T^{\epsilon}.
 \end{split}
\end{equation*}
Thus, substituting the integral estimate  in \eqref{First result},   we have 
 \begin{equation*}
\begin{split}
S_{1}(X) & =   O( X^{\frac{3}{2}+\epsilon} T^{\epsilon})  + O( X^{1+\epsilon}Y) +  O(X^{-A}).  
\end{split}
\end{equation*}
We substitute  $T = \frac{X^{1+\epsilon}}{Y}$ and choose  $Y =X^{\frac{1}{2}+\epsilon} $ to get 
 \begin{equation*}
\begin{split}
S_{1}(X) & =   O( X^{\frac{3}{2}+\epsilon}).
\end{split}
\end{equation*} 

\noindent
In the case of $r=2$, from the \eqref{Gen-formula}, we have 
 \begin{equation*}\label{SecondResult}
\begin{split}
S_{2}(X) 
& =  \underset{s= 2}{\rm Res} \left(R_{2}(s ) \tilde w(s)  \right)  + O\left(X^{\frac{3}{2}+\epsilon}  \int_{1}^{T} \frac{|L_{2}(3/2+\epsilon+it )|}{t}  dt \right)  \\
 & \quad +  O( X^{\frac{3}{2}+\epsilon})  + O( X^{1+\epsilon}Y) +  O(X^{-A})  
\end{split}
\end{equation*}
with $ L_{2}(s)  =   \zeta(s-1) L(s, \chi_{8}) L(s-1, sym^{2}f)  L(s, sym^{2}f \times \chi_{8}).$ Let
\begin{equation*}
 \begin{split}
I_{2}=   \int_{1}^{T} \frac{|L_{2}(3/2+\epsilon+it )|}{t}  dt.
\end{split}
\end{equation*} 
Substituting the decomposition of $ L_{2}(s)$ and using the absolute convergence of respective $L$-functions, we have
\begin{equation*}
 \begin{split}
I_{2} & \ll   \int_{1}^{T} \frac{| \zeta(1/2+\epsilon+it) L(1/2+\epsilon+it, sym^{2}f)|}{t}  dt.\\
\end{split}
\end{equation*}
Appealing the  dyadic division method and then the Cauchy-Schwarz inequality, we have
\begin{equation*}
 \begin{split}
I_{2} 
& \ll \log T \underset{1 \le T_{1} \le T} {\rm sup} \left(
\frac{1}{T_{1}}  \left( \int_{1}^{T_{1}}  {|\zeta(1/2+\epsilon+it)|}^{2}  dt\right)^{\frac{1}{2}}   \left( \int_{1}^{T_{1}}  {|L(1/2+\epsilon+it , sym^{2}f)|}^{2}  dt\right)^{\frac{1}{2}} \right) \\
&   \ll T^{\frac{1}{4}+\epsilon}.\\
\end{split}
\end{equation*}
which is obtained by using integral estimates for respective $L$-functions. Thus, substituting the estimate of $I_{2}$ in \eqref{SecondResult},  we have 
 \begin{equation*}
\begin{split}
S_{2}(X) & =  \underset{s= 2}{\rm Res} \left(R_{2}(s ) \tilde w(s)  \right)  + O\left(X^{\frac{3}{2}+\epsilon} T^{\frac{1}{4}+\epsilon} \right)  + O( X^{1+\epsilon}Y) +  O(X^{-A}).  
\end{split}
\end{equation*}
We substitute  $T = \frac{X^{1+\epsilon}}{Y}$ and choose  $Y =X^{\frac{3}{5}+\epsilon} $ to get 
 \begin{equation*}
\begin{split}
S_{2}(X) & = C X^{2} + O\left(X^{\frac{8}{5}+\epsilon}\right).
\end{split}
\end{equation*}
where $C$ is a positive absolute constant given by 
$$ C =  L(2, \chi_{8}) L(1, sym^{2}f)  L(2, sym^{2}f \times \chi_{8}) U_{2}(2).$$
This completes the proof.

\subsection{Proof of \thmref{Estimate-TR}:} For each fixed $r \in \mathbb{N}$, following the argument as in \S3.1, it is enough to obtain an estimate for the integral occurring in \eqref{Gen-formula}  get an estimate for the sum  $S_{r}(X)$ given in  \eqref{Sum-PIBQF}. Let
\begin{equation*}
 \begin{split}
I_{r}=   \int_{1}^{T} \frac{|L_{r}(3/2+\epsilon+it )|}{t}  dt.
\end{split}
\end{equation*} 
 We substitute $L_{r}(s)$ from \lemref{LDecomposition} to get 
 \begin{equation*}
 \begin{split}
I_{r} &=   \int_{1}^{T}  \frac{dt}{t  } \times  \displaystyle{\prod_{n=0}^{[r/2]}} 
\left(
{ \left|L \left( \frac{1}{2}+\epsilon+it , sym^{r-2n}f \right) L\left( \frac{3}{2}+\epsilon+it , sym^{r-2n}f \times \chi_{8} \right) \right|}^{\left({r \choose n}- {r \choose {n-1}}\right)}  
\right) \\
\quad & \ll\int_{1}^{T}  \frac{1}{t  } \times \displaystyle{\prod_{n=0}^{[r/2]}} {\left|L\left( \frac{1}{2}+\epsilon+it , sym^{r-2n}f \right)\right|}^{\left({r \choose n}- {r \choose {n-1}}\right)}  dt
\end{split}
\end{equation*} 
where we use the fact that $L(s , sym^{\ell}f)$ converges absolutely for $\Re(s) >1$ for each $\ell \ge 0$. We consider two cases when $r$ is even and $r$ is odd separately.

\smallskip
\noindent
\textbf{Case 1:}  Let $r$ is even, i.e., $r =2m$(say). Then 
 \begin{equation*}
 \begin{split}
I_{r} 
& \ll\int_{1}^{T}  \frac{1}{t  } \times| {\displaystyle{\prod_{n=0}^{m}} {L(1/2+\epsilon+it , sym^{2m-2n}f)}^{\left({2m \choose n}- {2m \choose {n-1}}\right)} |}  dt \\
& \ll  \underset{1 \le t \le T} {\rm max}\left( \zeta(1/2 +\epsilon +it)^{ \frac{1}{m} {2m \choose m-1}}  {L(1/2+\epsilon+it , sym^{2}f)}^{\frac{3}{m-1}{2m \choose m-2}}  \right) \\
& \qquad  \times \log T  \underset{1 \le T_{1} \le T} {\rm sup}\left( \frac{1}{T_{1}  }  \int_{1}^{T_{1}}   | {\displaystyle{\prod_{n=0}^{m-2}} {L(1/2+\epsilon+it , sym^{2m-2n}f)}^{\frac{2m-2n+1}{n}{2m \choose n-1}} }|  dt \right) \\
& \ll T^{\gamma_{r}+\epsilon}
\end{split}
\end{equation*} 
where we use the convexity/sub-convexity  bound of respective $L$-functions to get an upper estimate for $I_{r}$, and  $\gamma_{r} = \frac{13}{82m} {2m \choose m-1 } + \frac{15}{8(m-1)} {2m \choose m-2 } + \frac{1}{4} \left[ \displaystyle{\sum_{n=0}^{m-2}\frac{(2m-2n+1)^{2}}{n} {2m \choose n-1}} \right]$. 
We substitute the bound for $I_{r}$ in  \eqref{Gen-formula} to get (for each $r \in \mathbb{N}$)
 \begin{equation*}
\begin{split}
\sideset{}{^{\flat }} \sum_{n \le X}  \lambda_{f}^{r}(n)  \sigma_{1; \chi_{8}, \textbf{1}}(n) 
& =  \underset{s= 2}{\rm Res} \left(R_{r}(s ) \tilde w(s)  \right)  + O\left(X^{\frac{3}{2}+\epsilon}  T^{\gamma_{r}+\epsilon} \right)  \\
 & \quad +  O( X^{\frac{3}{2}+\epsilon})  + O( X^{1+\epsilon}Y) +  O(X^{-A}).  
\end{split}
\end{equation*}
We substitute  $T = \frac{X^{1+\epsilon}}{Y}$ and choose  $Y =X^{1-\frac{1}{2(1+\gamma_{r})}+\epsilon} $ to get 
 \begin{equation*}
\begin{split}
\sideset{}{^{\flat }} \sum_{n \le X}  \lambda_{f}^{r}(n)  \sigma_{1; \chi_{8}, \textbf{1}}(n) 
& = X^{2} P_{r}(\log X) +  O (X^{2-\frac{1}{2(1+\gamma_{r})}+\epsilon})  
\end{split}
\end{equation*}
where $P_{r}(t)$ is a polynomial of degree $d_{r} = \frac{1}{m} {2m \choose m-1} -1$ and $\gamma_{r}$
is given in \thmref{Estimate-TR}. This completes the proof for even $r$.

\bigskip
\noindent
\textbf{Case 2:} Let  $r$ is odd, i.e., $r =2m+1$(say). Then, first using the dyadic division method and Cauchy-Schwarz inequality to get  
 \begin{equation*}
 \begin{split}
I_{r} 
& \ll\int_{1}^{T}  \frac{dt}{t  } \times {\displaystyle{\prod_{n=0}^{m}}} {|L(1/2+\epsilon+it , sym^{2m+1-2n}f)|}^{\left({2m+1 \choose n}- {2m+1 \choose {n-1}}\right)}   \\
&  \ll  \log T  \underset{1 \le T_{1} \le T} {\rm sup} \begin{cases}
  \int_{1}^{T_{1}} \left( {|L(1/2+\epsilon+it , f)|}^{2\times \frac{2}{m}{2m+1 \choose m-1}}  dt\right)^{\frac{1}{2}} \times \\
   \left(  \frac{1}{T_{1}  }  \int_{1}^{T_{1}}  {\displaystyle{\prod_{n=0}^{m-1}}} {|L(1/2+\epsilon+it , sym^{2m+1-2n}f)|}^{2 \times \frac{2m+1-2n+1}{n}{2m \choose n-1}}   dt \right)^{\frac{1}{2}} \\
\end{cases} \\
&  \ll  \log T  \underset{1 \le T_{1} \le T} {\rm sup} \begin{cases}
 \underset{1 \le t \le T_{1}} {\rm max}\left( {|L(1/2+\epsilon+it , f)|}^{\frac{2}{m}{2m+1 \choose m-1} -1}  \right)  \int_{1}^{T_{1}} \left( {|L(1/2+\epsilon+it , f)|}^{2}  dt\right)^{\frac{1}{2}} \\
  \times \left(  \frac{1}{T_{1}  }  \int_{1}^{T_{1}}  \displaystyle{\prod_{n=0}^{m-1}} {|L(1/2+\epsilon+it , sym^{2m+1-2n}f)|}^{2 \times \frac{2m+1-2n+1}{n}{2m \choose n-1}}   dt \right)^{\frac{1}{2}}\\
\end{cases} \\
&\ll T^{\gamma_{r}+\epsilon}
\end{split}
\end{equation*}
which is obtained using convexity/sub-convexity bound and integral estimate of respective $L$-functions, and  $\gamma_{r} = \frac{2}{3m} {2m+1 \choose m-1 }  + \frac{1}{4} \left[ \displaystyle{\sum_{n=0}^{m-1}\frac{(2m+1-2n+1)^{2}}{n} {2m+1 \choose n-1}} \right]-\frac{5}{6}$. 
We know that $L_{r}(s)$ (for odd integer  $r$) does not have a pole. So, substituting the bound for $I_{r}$ in  \eqref{Gen-formula} to get
 \begin{equation*}
\begin{split}
\sideset{}{^{\flat }} \sum_{n \le X}  \lambda_{f}^{r}(n)  \sigma_{1; \chi_{8}, \textbf{1}}(n) 
& =   O\left(X^{\frac{3}{2}+\epsilon}  T^{\gamma_{r}+\epsilon} \right)  +  O( X^{\frac{3}{2}+\epsilon})   + O (Y^{2+\epsilon}) + O( X^{1+\epsilon}Y) +  O(X^{-A}). 
\end{split}
\end{equation*}
We substitute  $T = \frac{X^{1+\epsilon}}{Y}$ in above equation  and choose  $Y =X^{1-\frac{1}{2(1+\gamma_{r})}+\epsilon} $ to get 
 \begin{equation*}
\begin{split}
\sideset{}{^{\flat }} \sum_{n \le X}  \lambda_{f}^{r}(n)  \sigma_{1; \chi_{8}, \textbf{1}}(n) 
& =  O (X^{2-\frac{1}{2(1+\gamma_{r})}+\epsilon}).  
\end{split}
\end{equation*}
This completes the proof for odd $r$.

\smallskip
\textbf{Acknowledgement :}The author would like to thank IMSc, Chennai for providing financial support through institute fellowship.

\noindent
\section{Declarations:}
\noindent
\textbf{Ethical Approval:} Not applicable.\\
\textbf{Competing interests:} Not applicable.\\
\textbf{Author's contributions:} Not applicable.\\
\textbf{Funding:} Not applicable. \\
\textbf{Availability of data and materials:} This manuscript does not include any data. \\

\end{document}